\documentclass{article}
\usepackage[utf8]{inputenc}
\usepackage{amsmath,amssymb,amsthm}
\usepackage{color,graphicx,xy}
\input xy
\xyoption{all}
\title{The intersection form on the homology of a surface acted on by a finite group.}
\author{Jean Barge and Julien Marché}
\date{}
\newcommand{\Q}{\mathbb{Q}}
\newcommand{\C}{\mathbb{C}}
\newcommand{\Z}{\mathbb{Z}}
\newcommand{\R}{\mathbb{R}}
\newcommand{\F}{\mathbb{F}}
\renewcommand{\H}{\mathbb{H}}
\renewcommand{\tilde}{\widetilde}
\newtheorem{theorem}{Theorem}
\newtheorem{proposition}{Proposition}
\newtheorem{definition}{Definition}
\newtheorem{corollary}{Corollary}
\newtheorem{lemma}{Lemma}
\newtheorem{remark}{Remark}
\DeclareMathOperator{\Sp}{Sp}
\DeclareMathOperator{\Mod}{Mod}
\DeclareMathOperator{\U}{U}
\DeclareMathOperator{\SW}{SW}
\DeclareMathOperator{\Hom}{Hom}
\DeclareMathOperator{\End}{End}
\DeclareMathOperator{\GL}{GL}
\DeclareMathOperator{\arf}{Arf}
\DeclareMathOperator{\Ind}{Ind}
\DeclareMathOperator{\trd}{trd}
\DeclareMathOperator{\tr}{tr}
\DeclareMathOperator{\id}{Id}
\begin{document}

\maketitle

\section{Introduction}

The purpose of this note is to prove the following theorem.

\begin{theorem}
Let $G$ be a finite group acting freely on a compact oriented surface $S$ by homeomorphisms preserving the orientation.
There exists a $G$-invariant Lagrangian subspace in $H_1(S,\Z)$. 
\end{theorem}
A Lagrangian subspace is a free abelian subgroup $L$ of $H_1(S,\Z)$ of rank $g$, where $g$ is the genus of $S$ satisfying the following extra assumptions. It is a free summand, meaning that there exists $M\subset H_1(S,\Z)$ such that $L\oplus M=H_1(S,\Z)$, and it is isotropic, meaning that $x\cdot y=0$ for and $x,y\in L$ where $\cdot$ is the intersection pairing on $H_1(S,\Z)$. 

The topological meaning of this theorem is that $G$ can act in a unique way on the rational homology of a surface, together with its intersection form. 
\begin{corollary}
Let $G$ be a finite group acting freely on two compact oriented surfaces $S_1$ and $S_2$ of the same genus by homeomorphisms preserving the orientation.
There exists a $G$-equivariant linear map $H_1(S_1,\Q)\to H_1(S_2,\Q)$ preserving the intersection form. 
\end{corollary}
In this way, it is a symplectic straightening of a celebrated theorem of Chevalley and Weil stating that for any free action of $G$ on $S$ one has $H_1(S,\Q)=\Q^2\oplus \Q[G]^{2g-2}$ as $\Q[G]$-modules (see \cite{cw,gllm}). 
When $G$ is abelian, a complete classification of $G$-actions on surfaces has been obtained by Edmonds in \cite{ed}. 
This problem seems to have not been addressed before. Its proof uses a mixture of arguments coming from various fields, all of which are commonly used in surgery theory. Let us give an overview of its steps, which gives also the plan of the note. 

\begin{enumerate}
\item We define a Witt group $W_G^-(\Q)$ for symplectic $\Q$-vector space with an action of $G$: the theorem states that the class of $H_1(S,\Q)$ vanishes in $W_G^-(\Q)$ whatever be the group $G$ and its action on $S$. 

\item A bordism argument: the action of $G$ on $S$ defines a fundamental class $[S/G]$ in $H_2(G,\Z)$ and the class of $H_1(S,\Q)$ in $W_G^-(\Q)$ only depends on $[S/G]$. At this point we hoped to find an interesting map $H_2(G,\Z)\to W_G^-(\Q)$. The sequel consists in proving that this map vanishes. 

\item We decompose $H_1(S,\Q)$ into pieces corresponding the the irreducible rational representations of $G$. 

\item From the fact that $H_2(G,\Z)$ is a torsion group and that $W_G^-(\Q)$ only has $2$-primary torsion, we reduce the proof to the case of $2$-groups.

\item We invoke results on rational representations of $2$-group, precisely a theorem of Fontaine extending results of Roquette and Witt, to reduce the result to four families of $2$-groups for which the result is easily proved. 
\end{enumerate}

The question arose when studying representations of mapping class groups of surfaces on the homology of their finite covers. 
Consider $G$ as a subgroup of $\Mod(S)$, the mapping class group of $S$, and let $\Gamma$ be its normalizer. Elements of $\Gamma$ act on the quotient $S/G$, giving a map $\Gamma\to \Mod(S/G)$ whose kernel is $G$ and whose image has finite index. The problem of identifying the image of the symplectic representation $\rho:\Gamma\to \Sp(H_1(S,\Z))$ has deserved a lot of attention, see for instance \cite{looijenga,gllm}. 

A modest consequence of our result is that one can describe in any case the target group of $\rho$ as the group $\Sp_G(L\oplus L^*)$ where $L=\Q\oplus \Q[G]^{g-1}$. The question whether the image of $\rho$ is commensurable to the sugroup of integral points stays a major one in the theory of representations of the mapping class group, implying in particular the Ivanov conjecture, see \cite{pw}. 

To end this introduction, we would like to raise an analogy, pointed out by Mathieu Florence. Given a finite Galois extension $K$ of $\Q$ with group $G$, one may ask if one can find a normal basis (i.e. of the form $(g x)_{g\in G}$ for some $x\in K$) which is orthonormal for the trace form $\tr_{K/\Q}$. A theorem of Bayer and Lenstra shows that it holds if $G$ has odd order, see \cite{bl}. Both the statement and the proof show a deep analogy with the main result of this note.

{\bf Acknowledgements:} We would like to thank Arthur Bartels, Ian Hambbleton, Wolfgang Lück, Oscar Randal-Williams for sharing their expertise and Mathieu Florence for his support and interest. 

\section{The Witt group $W_G^-(\Q)$}

As the notation suggest, we could replace $\Q$ by any field whose caracteristic does not divide twice the cardinality of $G$. We refrain from doing this for the sake of brevity. 

\begin{definition}
A symplectic $\Q[G]$-module is a finite dimensional $\Q$-vector space $V$, endowed with a symplectic form and an action of $G$ which preserves the form.  

It is called metabolic if there exists a $G$-invariant Lagrangian  $L\subset V$. 
\end{definition}

We say that two symplectic $\Q[G]$-modules $V$ and $W$ are Witt-equivalent if one can find two metabolic $\Q[G]$-modules $E$ and $F$ such that $V\oplus E$ is isomorphic to $W\oplus F$. We will write $V\sim W$. 

For instance, denote by $-V$ the vector space $V$ with the opposite symplectic form and the same action. One check that $V\oplus(-V)$ is metabolic, which shows that the set of Witt-equivalence classes of symplectic $\Q[G]$-modules form a group with respect to the direct sum : we denote it by $W_G^-(\Q)$.

\begin{remark}
We can alternatively define a symplectic $\Q[G]$-module as a $\Q[G]$-module $V$ with a map $\omega_G:V\times V\to\Q[G]$ satisfying $\omega_g(gx,hy)=g\omega(x,y)h^{-1}$ and $\omega_G(y,x)=-\overline{\omega(x,y)}$ where we have set $\overline{g}=g^{-1}$. The form $\omega_G$ is constructed from a $G$-invariant symplectic form $\omega$ by putting 
$$\omega_G(x,y)=\sum_{g\in G} \omega(x,gy)g.$$
In this form, the Witt group $W_G^-(\Q)$ is isomorphic to the standard skew-symmetric Witt group of the ring with involution $\Q[G]$, usually denoted by $W^-(\Q[G])$. 
\end{remark}

The following lemma, adapted to our purposes, is the heart of the Witt cancellation theorem. 

\begin{lemma}\label{karoubi}
Let $V$ be a symplectic $\Q[G]$-module and $I,J$ be two $G$-invariant coisotropic subspaces, meaning that $I\subset I^\perp$ and $J\subset J^\perp$. 
Then $I^\perp/I\oplus -(J^\perp/J)$ is metabolic.
\end{lemma}
\begin{proof}
The proof consists in showing that the image $\Delta$ of the diagonal map $I^\perp\cap J^\perp\to (I^\perp/I)\oplus -(J^\perp/J)$ is a $G$-invariant Lagrangian. It is left to the reader. 
\end{proof}
As a consequence of this lemma, a symplectic module $V$ vanishes in $W_G^-(\Q)$ if and only if it is metabolic. Indeed, suppose that $V\oplus W$ contains a Lagrangian $L$ and $W$ contains a Lagrangian $\Lambda$. 
Take $I=V\oplus \Lambda$ and $J=L$ so that $I^\perp/I=V$ and $J^\perp/J=0$. Lemma \ref{karoubi} claims that $V$ is metabolic.

Suppose now that $V=L\oplus L'$ contains two $G$-invariant transverse Lagrangians. 
The symplectic form gives an isomorphism $L'\simeq L^*$ which has to be $G$-invariant. We say that the symplectic $\Q[G]$-module is hyperbolic. The following lemma states that any metabolic $\Q[G]$-module is hyperbolic. 
\begin{lemma}\label{lemme2}
If $L$ is a $G$-invariant Lagrangian in a symplectic $\Q[G]$-module $V$, there exists a $G$-invariant Lagrangian transverse to $L$.
\end{lemma}
\begin{proof}
Forgetting the $G$-action, it is known that the set of Lagrangians transverse to $L$ is naturally an affine vector space directed by $S^2L$. The stabilizer of $L$ in the symplectic group acts on this space by affine transformations. A $G$-invariant transverse Lagrangian is obtained by considering the barycenter of the $G$-orbit of any Lagrangian transverse to $L$. 
\end{proof}

\section{A bordism argument}

Suppose that $G$ acts on the right and freely on $S$: the quotient map $S\to S/G$ is a principal $G$-bundle, hence classified by a map $f:S/G\to BG$. Its fundamental class is the class $f_*([S/G])\in H_2(BG,\Z)=H_2(G,\Z)$. 

Let us rephrase this using covering spaces, supposing that the quotient $S/G$ is connected. It forces us to make use of the following more suggestive notation. 
Let $p:\tilde{S}\to S$ be Galois covering with group $G$ where $S$ is a connected compact oriented surface. Its monodromy is a morphism $\rho:\pi_1(S)\to G$ which determines the covering completely.

As $H_2(\pi_1(S),\Z)=H_2(S,\Z)$ is generated by the fundamental class $[S]$ we define the fundamental class of the covering to be the element $\rho_*([S])\in H_2(G,\Z)$.

\begin{proposition}
Given a finite group $G$, there exists a unique map 
$$\Phi_G:H_2(G,\Z)\to W_G^-(\Q)$$
such that for any Galois covering $p:\tilde{S}\to S$ with monodromy $\rho:\pi_1(S)\to G$, the class of $H_1(\tilde{S},\Q)$ in $W_G^-(\Q)$ is equal to $\Phi_G(\rho^*([S]))$.
\end{proposition}

\begin{proof}
Given any group $G$, we denote by $BG$ a $K(G,1)$-space and consider the bordism group $\Omega^2(BG)$. It is generated by pairs $(S,f)$ where $S$ is a compact oriented surface and $f:S\to BG$ is a continuous map, considered up to homotopy. Recall that if $S$ is connected, it is equivalent to give $f$ or to give a representation $\rho=f_*:\pi_1(S)\to G$ up to conjugation. 

We add the relation $(S_1\amalg S_2,f_1\amalg f_2)\sim (S_1,f_1)+(S_2,f_2)$ and for any compact oriented $3$-manifold $M$ with boundary and map $f:M\to BG$, 
$$(\partial M, f|_{\partial M})\sim 0.$$ 

It is well-known that the map $\Omega^2(BG)\to H_2(G,\Z)$ given by $(S,f)\mapsto f_*([S])$ is an isomorphism.
This can be proved by hand or as a by-product of the Atiyah-Hirzebruch spectral sequence. 

Using this point of view, the proposition reduces to showing that the map sending  $(S,\rho)$ to the Witt class of $H_1(\tilde{S},\Q)$ induces a linear map 
$\Phi_G:\Omega_2(BG)\to W_G^{-}(\Q)$. 

If $S$ is not connected, we define $\Phi_G(S,f)$ to be the sum of the image of its connected components. In this way, the first relation in $\Omega^2(BG)$ is trivially satisfied. We only need to show that if $M$ is a connected $3$-manifold with boundary and $\rho:\pi_1(M)\to G$ is a morphism, then $H_1(\tilde{S},\Q)$ is metabolic. For that, we consider the $G$-covering $\tilde{M}\to M$ corresponding to $\rho$. 

The natural map $H_1(\tilde{S},\Q)\to H_1(\tilde{M},\Q)$ is $G$-equivariant and its kernel is Lagrangian by Poincaré duality. Hence $H_1(\tilde{S},\Q)$ is metabolic as claimed.  
\end{proof}

\section{Decomposition of the homology of the cover}

\subsection{Homology with twisted coefficients}\label{twisted}
Let $p:\tilde{S}\to S$ be a galois $G$-cover, where $G$ acts on the right. Given any left $\Q[G]$-module $E$, we define $H_*(S,E)$ to be the homology of the complex $C_*(\tilde{S},\Q)\otimes_{\Q[G]}E$. 

Given two left $\Q[G]$-modules $E$ and $F$, there is an intersection product $H_1(S,E)\times H_1(S,F)\to H_0(S,E\otimes F)$ where $E\otimes F$ is the usual tensor product over $\Q$ endowed with the diagonal action of $G$. It is defined as follows: suppose that $\alpha,\beta$ are two 1-simplices in $\tilde{S}$ whose projections in $S$ are transversal. Denoting by $\epsilon_p(\alpha,\beta)$ the sign of the intersection of $\alpha$ and $\beta$ at $p$, we have:

\begin{equation}\label{intersection}
\alpha\otimes e \cdot \beta\otimes f=\sum_{g\in G}\sum_{p\in \alpha\cap \beta g} \epsilon_p(\alpha,\beta g) \,\,p\otimes e\otimes g^{-1} f.
\end{equation}

Our main example will be $\Q[G]$, viewed as a $\Q[G]$-bimodule: it gives the isomorphism of right $\Q[G]$-modules $H_1(\tilde{S},\Q)=H_1(S,\Q[G])$. Composing the above product with the pairing $\Q[G]\times \Q[G]\to \Q$ given by $\langle g,h\rangle=1$ if $g=h$ and $0$ else, gives back the intersection form 
$$ H_1(S,\Q[G])\times H_1(S,\Q[G])\to H_0(S,\Q)= \Q.$$

\subsection{Decomposition into irreducible representations}\label{decomposition}

Fix a finite group $G$ and consider a family $(\rho_i:G\to \GL(V_i))_{i\in I}$ representing every isomorphism class of irreducible rational finite dimensional representation.

By the Schur Lemma, the space of endomorphisms of $V_i$ commuting with $\rho_i(G)$ is a division algebra, noted $D_i$ (setting $\phi\psi=\psi\circ \phi$ for $\phi,\psi\in D_i)$. This endows $V_i$ with the structure of right $D_i$-module. 
Averaging a scalar product, we find that each $V_i$ has an invariant scalar product, giving a $G$-isomorphism $\phi_i:V_i\to V_i^*$. The dual space $V_i^*$ has a structure of left $D_i$-module hence, there exists an involution on $D_i$ such that $\phi_i(x\lambda)=\overline{\lambda}\phi_i(x)$. This data is better encoded into a Hermitian form $h_i:V_i\times V_i\to D_i$ satisfying for all $x,y\in V_i$ and $\lambda,\mu\in D_i$:
$$h_i(x\lambda,y\mu)=\overline{\lambda}h_i(x,y)\mu\text{ and }h_i(y,x)=\overline{h_i(x,y)}.$$

Let $S$ be a surface with a representation $\rho:\pi_1(S)\to G$: we may consider as in Section \ref{twisted} the homology $H_1(S,V_i)$. It has a structure of right $D_i$-module and the composition of the intersection form and the Hermitian form gives an skew-Hermitian form 
$$H_1(S,V_i)\times H_1(S,V_i)\to D_i.$$
These forms may be viewed as elements of a Witt group $W^-(D_i)$ whose definition is analogous to the definition of $W^-_G(\Q)$. As before, we can define a map $\Phi^i_G:H_2(G,\Z)\to W^-(D_i)$ such that for any $\rho:\pi_1(S)\to G$ one has $\Phi_G^i(\rho_*[S])=H_1(S,V_i)$. 

In this section, we prove the following reduction. 
\begin{lemma}\label{reduction}
If the maps $\Phi^i_G:H_2(G,\Z)\to W^-(D_i)$ vanish for all rational irreducible $G$-modules $V_i$ with Hermitian forms $h_i:V_i\times V_i\to D_i$, then $\Phi_G=0$. 
\end{lemma}

\begin{proof}
Recall the Artin-Wedderburn decomposition $$\Q[G]=\prod_{i\in I}\End_{D_i}(V_i)=\bigoplus_{i\in I}V_i\otimes _{D_i}V_i^*.$$
We write it in two ways to insist that it is both an isomorphism of algebras and of $\Q[G]-\Q[G]$-bimodules. 

The pairing $\langle \cdot,\cdot\rangle$ on $\Q[G]$ is $G$-invariant, hence it induces a $G$-invariant map $\Q[G]\to \Q[G]^*$. As all the $V_i$s are self-dual, this map must preserve the decomposition. In other terms, the decomposition is orthogonal. Moreover, a computation shows that the pairing in restriction to $\End_{D_i}(V_i)$ has to be 

\begin{equation}\label{adjoint}
\langle f,g\rangle=\frac{n_i}{|G|} \tr (fg^*)\quad \forall f,g\in \End_{D_i}(V_i)
\end{equation}

where $g^*$ is the adjonction with respect to the Hermitian form $h_i$, $n_i=\dim_{D_i}V_i$ and $V_i$ has to be though of as a $\Q$-vector space when taking the trace.

Consider now a morphism $\rho:\pi_1(S)\to G$. We can write 
$$H_1(S,\Q[G])=\bigoplus_{i\in I} H_1(S,\End(V_i))=\bigoplus_{i\in I} H_1(S,V_i)\otimes_{D_i} V_i^*.$$

In the last formula, we applied the universal coefficients theorem to write $H_1(S,V_i\otimes_{D_i}V_i^*)=H_1(S,V_i)\otimes_{D_i}V_i^*$, noting that the homology with coefficients only depends on the left $G$-action. 

By assumption, $H_1(S,V_i)$ vanishes in $W^-(D_i)$ hence there exist Lagrangians $L_i\subset H_1(S,V_i)$ for all $i\in I$. One can check from Equation \ref{adjoint} that the following space
$$L=\bigoplus_{i\in I}L_i \otimes_{D_i}V_i^*$$
is a $G$-invariant Lagrangian in $H_1(S,\Q[G])$, proving the lemma. 
\end{proof}

\section{Reduction to $2$-groups}
\subsection{Induction}
Suppose that $H$ is a subgroup of $G$ and we have a morphism $\rho:\pi_1(S)\to H$. We may compare $H_1(S,\Q[H])$ and $H_1(S,\Q[G])$ as symplectic $\Q[G]$-modules. 

To this aim, we observe that for any symplectic $\Q[H]$-module $V$, the induced module $\Ind V=V\otimes_{\Q[H]} \Q[G]$ has a natural symplectic structure when viewed as a direct sum of copies of $V$ indexed by representatives of $H\backslash G$. This construction gives a morphism $\Ind: W_H^-(\Q)\to W_G^-(\Q)$ and satisfies the following lemma, whose proof is left to the reader. 

\begin{lemma}\label{induction}
If $\rho:\pi_1(S)\to H$ is a morphism and $H$ is a subgroup of $G$, then 
$$H_1(S,\Q[G])=\Ind H_1(S,\Q[H]) \in W_G^-(\Q).$$
\end{lemma}

\subsection{Witt groups have 2-primary torsion}
We will need the following result:
\begin{theorem}\label{2torsion}
For any finite group $G$, the Witt group $W_G^-(\Q)$ has only 2-primary torsion. 
\end{theorem}

This statement seems well-known to specialists although we couldn't find a proof of it in the literature. 
By decomposing into irreducible pieces as in Section \ref{decomposition}, it reduces to the same statement for all $W^-(D_i)$ where the $D_i$'s are involved in $\Q[G]$. As such, it is proved by Wall in \cite{wall}:  we follow the formulation of Ranicki (Proposition 22.11 in \cite{ranicki}).
\begin{theorem}
Let $D$ be a division ring with involution such that $M_n(D)$ is a simple factor of $\Q[G]$ for some finite group $G$ and some integer $n$, then $W^-(D)$ is a countable abelian group of finite rank with countable $2$-primary torsion only. 
\end{theorem}
\begin{proof}
The proof consists in reducing the statement to the case of the standard Witt group of a field which is classical, see for instance \cite{milnor}. For the interested reader, we sketch the steps of the proof.

{\bf First case:} If $D$ is a field with trivial involution, then $W^-(D)=0$ as any symplectic form over a field has a Lagrangian. 

{\bf Second case:} If $D$ is a field with a non trivial involution, pick an element $i\in D$ such that $\overline{i}=-i$. The map $h\mapsto i h$ induces an isomorphism from $W^-(D)$ to $W^+(D)$. Let $K$ denote the subfield of $D$ preserved by the involution. Jacobson theorem tells that the map $h(x,y)\mapsto h(x,x)$ gives an embedding $W^+(D)\to W^+(K)$, hence the result follows from the case of $W^+(K)$.

{\bf Third case:} If $D$ is not commutative, it is a quaternion algebra over its center (mainly because $D$ is isomorphic to its opposite thanks to the involution, see for instance \cite{wall}). The proof reduces to the previous case by applying Jacobson theorem, see \cite{lewis}. 
\end{proof}
Given a finite group $G$ and a $2$-sylow $H$, we have the following commutative square:

$$\xymatrix{H_2(H,\Z)\ar[r]^{\alpha}\ar[d]_{\Phi_H}& H_2(G,\Z)\ar[d]_{\Phi_G}\ar@/_1pc/[l]_T\\W^-_H(\Q)\ar[r]^{\Ind}& W^-_G(\Q)}$$
Here $\alpha$ is the map induced by the inclusion $H\subset G$ and $T$ is the transfer map. Let $2d+1$ be the index of $H$ in $G$: we have $\alpha\circ T=(2d+1)\id$.  

Suppose that $\Phi_H$ is trivial. Given $x\in H_2(G,\Z)$, the equation above and the commutativity of the diagram gives $(2d+1)\Phi_G(x)=0$. By Theorem \ref{2torsion}, multiplying by an odd number is injective in $W_G^-(\Q)$, hence $\Phi_G(x)=0$. 
This shows that if $\Phi_G$ is trivial for all $2$-groups $G$, then it is trivial for any finite group.

\section{The case of $2$-groups}

\subsection{Special 2-groups}\label{special}
We consider the following list of $2$-groups together with their unique rational irreducible and faithful representation $E$. They represent all possible extensions of $\Z/2\Z$ by $\Z/2^{n-1}\Z$. 

\subsubsection{Cyclic group}
The group $C2^n=\langle x|x^{2^n}=1\rangle$ ($n\ge 1$) acts on $E=\Q(\zeta)$ by $xu=\zeta u$ where $\zeta$ has order $2^n$. The action commutes with the involution defined by $\overline{\zeta}=\zeta^{-1}$. 

\subsubsection{Dihedral group}
In the remaining examples, we choose a root $\zeta$ of order $2^{n-1}$. The group $D2^n=\langle x,y|x^{2^{n-1}}=y^2=1,yxy^{-1}=x^{-1}\rangle$ ($n\ge 2$) acts on $\Q(\zeta)$ by $xu=\zeta u$ and $yu=\overline{u}$. Viewing $\Q(\zeta)$ as a 2-dimensional vector space $E$ over $\Q(\zeta+\zeta^{-1})$ gives the unique faithful rational representation of $D2^n$. It preserves a quadratic form over $\Q(\zeta+\zeta^{-1})$ for which $1,i$ is an orthogonal basis. 

\subsubsection{Semi-dihedral group}
The group $SD2^n=\langle x,y|x^{2^{n-1}}=y^2=1,yxy^{-1}=x^{2^{n-2}-1}\rangle$ ($n\ge 4$) acts on $\Q(\zeta)^2$ by $x(u,v)=(\zeta u,-\zeta^{-1}v)$ and $y(u,v)=(v,u)$, preserving the standard Hermitian form. Setting $E$ to be the $\Q(\zeta-\zeta^{-1})$-vector space generated by the $SD2^n$ orbit of $(1,1)$, we get a faithful representation defined over $\Q(\zeta-\zeta^{-1})$ preserving the standard Hermitian form over that field. 
 
\subsubsection{Generalized quaternion group}
The group $Q2^n=\langle x,y|x^{2^{n-2}}=y^2=(xy)^2\rangle$ ($n\ge 3$) acts on $\Q(\zeta)^2$ by $x(u,v)=(\zeta u,\zeta^{-1}v)$ and $y(u,v)=(-v,u)$. The image of $\Q(\zeta+\zeta^{-1})[Q2^n]$ in $M_2(\Q(\zeta))$ is isomorphic to the quaternion algebra $E=\H\otimes \Q(\zeta+\zeta^{-1})$ where $\H$ is the standard quaternion algebra over $\Q$. The action of $Q2^n$ by left multiplication is the unique irreducible faithful representation of $Q2^n$.

\subsection{General $2$-groups}

Let us recall the following theorem due to Fontaine, see \cite[Theorem 1.3]{hm}.
\begin{theorem}\label{fontaine}
Let $G$ be a $2$-group and $V$ be an irreducible $\Q[G]$-module. Then there exists subgroups $K\lhd H$ of $G$ such that $H/K$ is isomorphic to one of the 4 families of groups of Section \ref{special} and $V=\Q[G]\otimes_{\Q[H]} E$ where $E$ is the corresponding faithful irreducible representation, viewed as a $\Q[H]$-module through the quotient map $H\to H/K$. 
\end{theorem}
Let us prove now that $\Phi_G=0$ for any finite $2$-group $G$. From Lemma \ref{reduction}, it is sufficient to prove that  $H_1(S,V)=0$ in $W^-(D)$ for any map $\pi_1(S)\to G$ and any irreducible rational representation $V$ of $G$ with division algebra $D$.

Let $H,K,E$ be as in Fontaine Theorem and denote by $D$ the division algebra involved in $E$. One has 
$$H_1(S,V)=H_1(S,\Q[G]\otimes_{\Q[H]} E)=H_1(S',E)$$
where $S'=\tilde{S}/H$ is the intermediate covering of $S$ with Galois group $H$. Moreover, this isomorphism preserves the skew-Hermitian forms and the homology group $H_1(S',E)$ can defined directly through the composition $\pi_1(S')\to H\to H/K$.

In particular, we can suppose that $G$ is one of the groups in the list and $E$ is the corresponding irreducible faithful representation. 

In the cases $G=C2^n,SD2^n$ and $Q2^n$ we have $H_2(G,\Z)=0$ which proves that $H_1(S',E)$ vanishes in $W^-(D)$. 

Let us treat the remaining case of $G=D2^n$. In this case $D=\Q(\zeta+\zeta^{-1})$ with trivial involution. The group $W^-(D)$ is then trivial as any symplectic vector space over a field has a Lagrangian.

\end{document}